\documentclass[11pt]{amsart}
\usepackage[utf8]{inputenc}
\usepackage{bm}
\usepackage{fontenc}
\usepackage{amsfonts}
\usepackage{amssymb}
\usepackage{amsmath}
\usepackage{amsthm}\usepackage{mathtools}
\usepackage{enumerate}
\usepackage{enumitem} 
\usepackage[pagebackref,colorlinks,linkcolor=blue,citecolor=blue,urlcolor=blue,hypertexnames=true]{hyperref}
\usepackage[pagebackref,hypertexnames=true]{hyperref}
\usepackage{enumitem}
\usepackage{mathrsfs}
\usepackage{tikz}
\usetikzlibrary{calc}
\usepackage{marginnote}
\usepackage{xcolor,enumitem}
\usepackage{soul}\usepackage{tikz}
\usepackage[all,cmtip]{xy} \usepackage{caption}

\newcommand{\cstarx}[1]{\mathcal{#1}}

\newcommand{\N}{\mathbb{N}}
\newcommand{\C}{\mathbb{C}}

\newcommand{\cE}{\mathcal{E}}

\newcommand{\eps}{\varepsilon}

\usepackage{bbm}
\newcommand{\Hawaii}{Hawai\kern.05em`\kern.05em\relax i}

\newcommand{\SOTh}{\mathrm{SOT}\text{-}}

\newcommand{\NN}{\mathbb{N}}

\newcommand{\cstu}{\mathrm{C}^*_u}
\newcommand{\cstql}{\mathrm{C}^*_{\!\mathit{ql}}}

\newtheorem*{rigprob*}{Rigidity Problem for uniform Roe Algebras}
\newtheorem*{rigprobcorona*}{Rigidity Problem for uniform Roe Coronas}

\newcommand{\cstar}{$\mathrm{C}^*$}

\newcommand{\cB}{\mathcal{B}}
\newcommand{\cK}{\mathcal{K}}

\newcommand{\B}{\cstarx{B}}

\newcommand{\R}{\mathbb{R}}

\numberwithin{equation}{section}

\newtheorem{theorem}{Theorem}[section]
\newtheorem*{theorem*}{Theorem}

\newtheorem{problem}[theorem]{Problem}
\newtheorem*{proposition*}{Proposition}
\newtheorem{lemma}[theorem]{Lemma}
\newtheorem*{lemma*}{Lemma}

\newtheorem*{corollary*}{Corollar}

\newtheorem*{fact*}{Fact}
\theoremstyle{definition}

\newtheorem*{definition*}{Definition}

\newtheorem*{acknowledgments}{Acknowledgments}
\newtheorem{claim}[theorem]{Claim}
\newtheorem*{claim*}{Claim}

\newtheorem*{conjecture*}{Conjecture}

\theoremstyle{remark}

\newtheorem*{example*}{Example}
\newtheorem{remark}[theorem]{Remark}
\newtheorem*{remark*}{Remark}

\newtheorem*{note*}{Note}
\newtheorem*{question*}{Question}

\newcommand{\norm}[1]{\left\lVert #1 \right\rVert}

\DeclareMathOperator{\Id}{Id}

\DeclareMathOperator{\propg}{prop}

\usepackage{enumitem}

\numberwithin{equation}{section}

\begin{document}

\title[Quasi-local algebra of expander graphs]{A note on the quasi-local algebra of expander graphs}%

\date{\today} %

\author[Braga]{Bruno M. Braga}
\address[B. M. Braga]{IMPA, Estrada Dona Castorina 110, 22460-320, Rio de Janeiro, Brazil}
\email{demendoncabraga@gmail.com}
\urladdr{\url{https://sites.google.com/site/demendoncabraga}}

\author[\v{S}pakula]{J\'{a}n \v{S}pakula}
\address[J. \v{S}pakula]{School of Mathematical Sciences, University of Southampton, Highfield, Southampton,
SO17 1BJ, United Kingdom}
\email{jan.spakula@soton.ac.uk}
\urladdr{\url{https://spakula.github.io/}}

\author[Vignati]{Alessandro Vignati}
\address[A. Vignati]{
Institut de Math\'ematiques de Jussieu (IMJ-PRG)\\
Universit\'e Paris Cit\'e\\ Institut Universitaire de France\\
B\^atiment Sophie Germain\\
8 Place Aur\'elie Nemours \\ 75013 Paris, France}
\email{alessandro.vignati@imj-prg.fr}
\urladdr{\url{https://www.automorph.net/avignati}}

\maketitle

\begin{abstract}
We show that the   quasi-local algebra of a coarse disjoint union of expander graphs does not contain a Cartan subalgebra isomorphic to $\ell_\infty$.
N.~Ozawa has recently shown that these algebras are distinct from the uniform Roe algebras of expander graphs, and our result describes a further difference.
\end{abstract}  

\section{Introduction}
This paper focuses on the study of Roe-like algebras of metric spaces arising from expander graphs and their properties. We will in particular focus on quasi-local algebras associated to such spaces   and on their Cartan algebras.

In a recent breakthrough, N. Ozawa showed that the quasi-local algebra  of a coarse disjoint union of expander graphs is strictly larger than its uniform Roe algebra   (\cite[Corollary C]{Ozawa2023uRaSmallerQL}). Up until this result, the question of whether these algebras coincided was one of the important open problems in the field (see \cite{SpakulaTikuisis2019,SpakulaZhang2020JFA,Li:2021aa,KhukhroLiVigoloZhang2021AdvMath,BaudierBragaFarahVignatiWillett2024vNa}) which went back to  J.~Roe (see \cite[Page 20]{Roe1996Book}). In his proof, Ozawa showed that uniform Roe algebras can never contain the product of matrix algebras $\prod_{n\in\mathbb{N}}\mathrm M_n(\mathbb C)$ while  the quasi-local algebra of expander graphs always  contains this product. This result suggests a new question: how distinct can the quasi-local and the uniform Roe algebra be? The goal of the present note is to show that the quasi-local algebra of a coarse disjoint union of expander graphs does not have a Cartan subalgebra isomorphic to $\ell_\infty$ (see Theorem \ref{ThmExpNoCartanlInfty}), a property shared by all uniform Roe algebras.  The proof of this result relies both on earlier work on Cartan subalgebras of Willett and White (\cite{WhiteWillett2017}) as well as on Ozawa's (\cite{Ozawa2023uRaSmallerQL}); as Ozawa, we use the concentration of measure phenomenon in order to rule out the existence of such subalgebras. 

Let us carefully introduce our objects.  In coarse geometry, there are several  algebras of bounded operators which serve as models to code in \cstar-algebraic terms the large scale geometry of metric spaces; these are known as Roe-like algebras. Important for us are two of these, the \emph{uniform Roe} and the \emph{quasi-local algebra}, which we shall now define. Firstly, as coarse geometry is the study of metric spaces from afar, local properties of metric spaces are irrelevant in this context and one can restrict themselves to discrete spaces. In fact, throughout these notes, we will only deal with \emph{uniformly locally finite} (abbreviated as \emph{u.l.f.})\ metric spaces, i.e., spaces where balls of a given radius are finite and uniformly bounded in cardinality.

Given a u.l.f.\ metric space $(X,d)$, $\ell_2(X)$ denotes the Hilbert space of square-summable functions $X\to \C$ and $\cB(\ell_2(X))$ the algebra of bounded linear operators on $\ell_2(X)$. Given an operator $a\in \cB(\ell_2(X))$, its   \emph{propagation}  is given by 
\[\propg(a)=\sup\{d(x,y)\mid \langle a\delta_x,\delta_y\rangle\neq 0\} ,\]
where $(\delta_x)_{x\in X}$ denotes the canonical orthonormal basis of $\ell_2(X)$; so, $\propg(a)$ is in $ [0,\infty]$. The \emph{uniform Roe algebra of $X$}, denoted by $\cstu(X)$, is the norm closure of all operators in $\cB(\ell_2(X))$ whose propagation is finite. Identifying $\ell_\infty(X)$, the \cstar-algebra of bounded functions $X\to \C$, with the (multiplication) operators on $\ell_2(X)$ with propagation zero, i.e. the  operators which are diagonal  with respect to the basis $(\delta_x)_{x\in X}$, it is clear that $\ell_\infty(X)\subseteq \cstu(X)$. In fact, $\ell_\infty(X)$ is a \emph{Cartan subalgebra of $\cstu(X)$}, in the sense that $\ell_\infty(X)$ is a maximal abelian subalgebra (abbreviated as \emph{masa}), the  normalizer of $\ell_\infty(X)$ in $\cstu(X)$ generates the whole $\cstu(X)$ as a \cstar-algebra, and there is a faithful conditional expectation $\cstu(X)\to \ell_\infty(X)$  (see \cite[Proposition 4.1]{WhiteWillett2017}). 

It is straightforward that, for any operator in $\cstu(X)$, its cutoffs by projections in $\ell_\infty(X)$ given by subsets of $X$ which are far enough apart have small norm. More precisely, given $A\subseteq X$, let $\chi_A$ denote the orthogonal projection of $\ell_2(X)$ onto $\ell_2(A)$. Then, for any $a\in \cstu(X)$, the following holds: 
\begin{equation}\label{Eq.QL.property}
\forall \eps>0\ \exists r>0\ \forall A,B\subseteq X\ \text{if } d(A,B)>r \text{ then } \|\chi_Aa\chi_B\|\leq \eps. 
\end{equation} 
Indeed, if $b$ is a finite propagation operator with $\|a-b\|\leq \eps$, then $r=\propg(b)$ has the desired property, since $\|\chi_Ab\chi_B\|=0$ for all $A,B\subseteq X$ with $d(A,B)>r$.
It is readily seen that the subset of all operators in $\cB(\ell_2(X))$ satisfying \eqref{Eq.QL.property} forms a \cstar-subalgebra of $\cB(\ell_2(X))$. We call this \cstar-algebra the (uniform) \emph{quasi-local algebra of $X$} and denote it by $\cstql(X)$.

By the discussion above, we have that $\cstu(X)\subseteq \cstql(X)$ and, as mentioned in the first paragraph of this introduction, it was a long standing open problem whether these algebras were actually the same, regardless of $X$. This is an important problem: in fact, in case quasi-local operators correspond to the operators that can be approximated by finite propagation operators, one has  a method for checking whether an operator belongs to a uniform Roe algebra merely by estimating norms of off-diagonal block restrictions of the operator, without having to explicitly produce finite propagation approximants. This is crucial, for instance, in the work of Engel (see \cite{Engel:2014tx,Engel:2018vm}) on the index theory of pseudo-differential operators. In general, checking whether an operator is quasi-local is much easier than checking whether it belongs to $\cstu(X)$.

As of now, the largest class of u.l.f.\ metric spaces for which it is known that  $\cstu(X)=\cstql(X)$ is the class of spaces satisfying Yu's property A (see \cite[Theorem 3.3]{SpakulaZhang2020JFA}, and \cite[Theorem 5]{Ozawa2023uRaSmallerQL} for a much shorter proof). On the other hand, as proved by Ozawa, these algebras are never the same if $X$ is the coarse disjoint union of expander graphs (\cite[Corollary C]{Ozawa2023uRaSmallerQL}). Recall that for $k\in \N$ and $h>0$, a finite (undirected) graph $G=(V,E)$ is a \emph{$(k,h)$-expander} if every vertex is incident to at most $k$ edges, and for all $A\subseteq V$ we have that 
\[A\leq |V|/2 \text{ implies } |\partial A|\geq h|A|,\]
where $\partial A=\{v\in V\setminus A\mid \exists u\in A, \ (v,u)\in E\}$. This condition readily implies that $G$ is connected, and thus we consider $V$ as a metric space endowed with the shortest path distance. Let now $(X_n,d_n)_{n\in\mathbb{N}}$ be a sequence of finite metric spaces, each being a vertex set of some $(k,h)$-expander graph, for $k\in\N$ and $h>0$ fixed and independent of $n\in\N$, and such that $|X_n|\to \infty$ as $n\to \infty$. We endow $X=\bigsqcup_{n\in\mathbb{N}} X_n$ with a metric $d$ such that
\begin{enumerate}
    \item\label{Item.Metric.p.cdu.1} $d\restriction X_n\times X_n=d_{n}$ for all $n\in\N$, and 
    \item \label{Item.Metric.p.cdu.2} $d(X_n,X_m)\to \infty$ as $n,m\to \infty$ with $n\neq m$. 
\end{enumerate}
We refer to such a disjoint union $X$ as a \emph{coarse disjoint union of expander graphs}. Note that while the metric $d$ above is by no means unique,   properties \eqref{Item.Metric.p.cdu.1} and \eqref{Item.Metric.p.cdu.2} above characterize $d$ coarsely. %
For details on expander graphs, we refer to \cite{LubotzkyBook2010}.

We want to understand the possible differences between uniform Roe and quasi-local algebras in the setting of coarse disjoint union of expander graphs, with a specific focus on Cartan \cstar-subalgebras. 
While $\ell_\infty(X)$ is always a Cartan subalgebra of $\cstu(X)$,  in the case of expander graphs, not only $\ell_\infty(X)$ is not a Cartan subalgebra of $\cstql(X)$, but in fact $\cstql(X)$ admits no Cartan subalgebra isomorphic to $\ell_\infty$. More precisely,  we prove the following in Section \ref{SectionMainResult}. 

\begin{theorem}\label{ThmExpNoCartanlInfty}
    Let $X=\bigsqcup_{n\in\mathbb{N}}X_n$ be the coarse disjoint union of   expander graphs. Then $\cstql(X)$ does not have a Cartan subalgebra isomorphic to $\ell_\infty$.
\end{theorem}

As Ozawa's aforementioned result, Theorem \ref{ThmExpNoCartanlInfty} is also valid for \emph{asymptotic expanders}, a weaker property than being an expander. We discuss this is Remark \ref{Remark.Asymp.Exp} below.

We note that any \cstar-subalgebra of $\cstql(X)$ which is abelian and strongly closed is of the  form $\ell_\infty(I)$ for some countable, possibly finite, set $I$ (see \cite[Theorem 1.4]{BaudierBragaFarahVignatiWillett2024vNa}). Therefore, Theorem \ref{ThmExpNoCartanlInfty} implies that the quasi-local algebra of a coarse disjoint union of expander graphs cannot have a strongly closed Cartan subalgebra.  
The following problem remains open.

\begin{problem}
Let $X=\bigsqcup_{n\in\mathbb{N}}X_n$ be the coarse disjoint union of   expander graphs. Can $\cstql(X)$ have a  Cartan subalgebra?%
\end{problem}

As our motivation is to further distinguish $\cstu(X)$ from $\cstql(X)$, we point out that it remains open whether the uniform Roe algebras of u.l.f.\ metric spaces $X$ and $Y$ are isomorphic if and only if their quasi-local algebras are isomorphic. While we cannot solve this question, we prove that the forward implication is always true. In fact, the following stronger statement is proved in Section \ref{SectionExtIso}.

\begin{theorem}\label{ThmIsoOfURAExtendsToQL}
Let $X$ and $Y$ be  u.l.f.\ metric spaces and $\Phi\colon \cstu(X)\to \cstu(Y)$ be an isomorphism. Then, $\Phi$ canonically extends to an isomorphism between $\cstql(X)$ and $\cstql(Y)$. 
\end{theorem}

We leave untouched the problem whether $\cstql(X)\cong \cstql(Y)$ implies $\cstu(X)\cong\cstu(Y)$. We note that this is related to the so called \emph{strong rigidity problem for quasi-local algebras}, which remains open to date. Indeed, while it was shown in \cite[Theorem 3.5]{BaudierBragaFarahKhukhroVignatiWillett2021uRaRig} that if $\cstql(X)$ and $\cstql(Y)$ are isomorphic, then $X$ and $Y$ are coarsely equivalent. It is still unknown whether this hypothesis is strong enough to imply that $X$ and $Y$ are \emph{bijectively} coarsely equivalent. A positive answer to this problem  would immediately imply that $\cstu(X)$ and $\cstu(Y)$ must be isomorphic given that $\cstql(X)$ and $\cstql(Y)$ are isomorphic.

We remark that in the non-uniform setting, the rigidity problems have been completely solved by Mart\'inez and Vigolo \cite[Corollary 8.3.5]{martnez2024rigidity}, and the non-uniform analogue of Theorem \ref{ThmIsoOfURAExtendsToQL} follows from \cite[Theorem 10.2.1]{martnez2024rigidity}; see also \cite[Theorem 4.5]{MartinezVigolo2025PMIHES}.

\section{Expander graphs and Cartan  subalgebras}\label{SectionMainResult}

The proof of Theorem \ref{ThmExpNoCartanlInfty} will follow the idea of the proof of \cite[Theorem B]{Ozawa2023uRaSmallerQL}. However, we need a slightly stronger statement than in \cite{Ozawa2023uRaSmallerQL}, see Theorem \ref{ThmExpContainsPropducMatrixPropDim} below. Thus for completeness, we recall some definitions and the statement of the concentration of measure phenomenon. Given $n\in\N$, $\R^n$ denotes the Euclidean space endowed with its canonical norm, $\partial B_{\R^n}$ its unit sphere, and we let $\sigma_n$ be the unique probability measure on $\partial B_{\R^n}$ which is invariant under orthogonal transformations, i.e., if $u\colon \R^n\to \R^n$ is an isometry preserving the origin, then $\sigma_n(A)=\sigma_n(u(A))$ for all measurable $A\subseteq \partial B_{\R^n}$.  Following common terminology in the literature, as the measures $\sigma_n$ are probability measures, we write $\mathbb P$ for $\sigma_n$ if $n\in\N$ is clear  from the context. For a complete proof of the next theorem, we refer the reader to Theorem 12.2.2 in the monograph \cite{AlbiacKalton2006}.

\begin{theorem}[Concentration of measure phenomenon]   \label{ThmConcentrationMeasure}
Let $n\in\N$ and $f\colon \partial B_{\R^n}\to \R$ be $1$-Lipschitz. Then, for all $t>0$, we have 
\[\mathbb P\left( \xi\in \partial B_{\R^n}\mid |f(\xi)-\mathbb Ef|>t\right)\leq 4e^{-\frac{t^2n}{72\pi^2}}.\]
\end{theorem}

The next lemma is a version of \cite[Lemma 4]{Ozawa2023uRaSmallerQL}. For a subspace $V\subseteq \R^n$, we endow its unit sphere $\partial B_V$ with its canonical spherical probability measure given by an isometry $V\cong \R^{n\gamma}$. As we need to treat the complex case as well, for a (complex) subspace $V\subseteq\C^n$, we endow its unit sphere $\partial B_{V}$ with a probability measure by identifying $V$ with $\R^{2\dim_{\C}(V)}$ and referring to the real case.
We write $\mathbb E_{V} $ to denote the expectation over $\zeta\in \partial B_{V}$.

The following notation will be used: given $k,n\in \N$ and $\bar\xi=(\xi_1,\ldots, \xi_k)$ in $(\partial B_{\R^n})^k$ (or $(\partial B_{\C^n})^k$), we let $p_{\bar \xi}$ denote the orthogonal projection of $\R^n$ (or $\C^n$ respectively) onto the real (respectively, complex) $\mathrm{span}\{\xi_1,\ldots, \xi_k\}$. Henceforth, we employ the convention that the implicit scalars match the type of the object, unless explicitly stated otherwise.

\begin{lemma}\label{LemmaMedellin12Dec23} For all $k\in \N$, $0< \delta < 1/10$ and $\delta<\gamma<1$, there are $c,L>0$ and $n_0\in \N$ such that for all $n\in\N$ with $n\geq n_0$ and all subspaces $V\subseteq \R^n$ or $V\subseteq \C^n$ with $\dim(V)\geq n\gamma$ we have \begin{align*}\label{EqMedellin12Dec23}
    \mathbb P\left(\bar\xi\in (\partial B_{V})^k\mid \max_{\substack{  A\subseteq \{1,\ldots, n\}\\ |A|\leq n \delta}} \|\chi_Ap_{\bar\xi}\|\geq 14\pi  \sqrt{({\delta}/{\gamma})\log\left({1}/{\delta}\right)k} \right)\leq Le^{-cn}
\end{align*}
\end{lemma}

\begin{proof}
For simplicity,  we assume throughout that  $n\delta$ and $n\gamma$ are integers.
Similarly, we restrict our estimates to subsets  $A\subseteq\{1,\ldots,n\}$ with exactly $n\delta$ elements.

We shall work in the real setting first, i.e.~we assume $V\subseteq \R^n$.

\begin{claim}
    $\mathbb E_{V} \|\chi_A\zeta\|$ is asymptotically at most $\sqrt{\delta/\gamma}$  as $n\to \infty$.
    
\end{claim}

\begin{proof}
Let $p_V$ be the orthogonal projection of $\R^n$ onto $V$,  $W=\mathrm{range}(p_V\chi_A)$, and $p_W$ be the orthogonal projection of $\R^n$ onto $W$.  Then, given $\zeta\in \partial B_V$, we have that 
$\|\chi_A\zeta\|\leq \|p_W\zeta\|$. Indeed, to check this, let $p_\zeta$ be the projection of $\R^n$ onto $\R \cdot \zeta$ and notice that, since $p_\zeta=p_\zeta p_V$, we have  
\[\|\chi_A\zeta\|=\|p_\zeta\chi_A\|=\|p_\zeta p_V\chi_A\|\leq \|p_\zeta p_W\|=\|p_W\zeta\|.\]
In particular, this implies that 
\[ \mathbb E_{V} \|\chi_A\zeta\|\leq \mathbb E_{V} \|p_W\zeta\|.\]
By the definition of $W$, it is clear that $\dim(W)\leq |A|=n\delta$. Therefore, $W$ is a subspace of $V$ of dimension at most $(\delta/\gamma)\dim(V)$. By \cite[Lemma 15.2.2]{MatousekBook2002}, $\mathbb E_{V} \|p_W\zeta\|$ is asymptotically at most $\sqrt{\delta/\gamma}$  as $n\to \infty$.\footnote{More precisely, this follows from the last line in the proof of \cite[Lemma 15.2.2]{MatousekBook2002}. }    
\end{proof}

By the previous claim, for all sufficiently large $n\in\mathbb{N}$, we have $\mathbb E_{V} \|\chi_A\zeta\|\leq 2\sqrt{\delta/\gamma}$. Let $\eps=2 \sqrt{({\delta}/{\gamma})\log(1/\delta)}$. Note that $\sqrt{\delta/\gamma} <\eps/2$, hence $\mathbb E_{V} \|\chi_A\zeta\|\leq \eps \leq \pi\eps$
for sufficiently large $n\in\N$. The measure concentration phenomenon (Theorem \ref{ThmConcentrationMeasure})  then  implies 
\begin{align}
        \mathbb P\left(\xi\in \partial B_{V}\mid \|\chi_A\xi\|>{7\pi \eps}\right)&\leq \mathbb P\left(\xi\in \partial B_{V}\mid \|\chi_A\xi\|>\mathbb E_{V}\|\chi_A\zeta\|+{6\pi \eps}\right)  \notag \\
        &\leq \mathbb P\left(\xi\in \partial B_{V}\mid |\|\chi_A\xi\|-\mathbb E_{V}\|\chi_A\zeta\||>{6\pi \eps} \right)\notag \\
        & <4e^{-\frac{\eps^2n\gamma}{2 }}
        \label{eq:l22-one-vector-estimate}
\end{align}
for all  $n\in\N$.

\def\cx{\operatorname{cx}}

We now argue that \eqref{eq:l22-one-vector-estimate} holds in the complex case as well. Assume that $V\subseteq \C^n$. We apply the real case, with $2n$ in place of $n$, as follows. Consider the natural isomorphism $\cx:\R^{2n}\cong\C^n$. We observe that for $\xi\in\R^{2n}$ and $A\subseteq\{1,\dots,n\}$, we have $\|\chi_{\tilde{A}}\xi\| = \|\chi_{A}\cx(\xi)\|$, where $\tilde{A}=\{2k-1,2k\colon k\in A\}$ denotes the basis labels that `correspond' to $A$ under $\cx^{-1}$. Thus the `complex' \eqref{eq:l22-one-vector-estimate} follows from the `real' \eqref{eq:l22-one-vector-estimate} applied to $\cx^{-1}(V)$.

The remainder of the proof applies to both $\R$ and $\C$ at the same time.
Since there are $\binom{n}{n\delta}$ subsets $A\subseteq\{1,\ldots, n\}$ with   $n\delta$ elements, it follows that 
    \begin{align*}
    \mathbb P\left(\xi\in \partial B_{V}\mid \max_{\substack{  A\subseteq \{1,\ldots, n\}\\ |A|\leq n \delta}} \|\chi_A\xi\|\geq  7\pi \eps  \right)\leq   4  \binom{n}{n\delta}  e^{-\frac{\eps^2n\gamma}{2 }}.
\end{align*}
By the definition of $\eps$, a computation using the inequality $\binom{n}{k}\leq (ne/k)^k$ guarantees that   $\binom{n}{n\delta}\leq  e^{ \frac{\eps^2n\gamma}{4}}$. Hence, 
\begin{equation}\label{EqQueroGranizado}
    \mathbb P\left(\xi\in \partial B_{V}\mid \max_{\substack{  A\subseteq \{1,\ldots, n\}\\ |A|\leq n \delta}} \|\chi_A\xi\|\geq  7\pi\eps\right)\leq   4    e^{-\frac{\eps^2n\gamma}{4}}
    \end{equation}
   for all $n\in\N$.

Notice that if $\bar\xi=(\xi_1,\ldots, \xi_k)\in (\partial B_{V})^k$ is an orthonormal tuple, then $p_{\bar\xi}=\sum_{i=1}^kp_{\xi_i}$. Therefore, if $\|\chi_Ap_{\bar\xi}\|\geq 7\pi\eps\sqrt{k}$, there must be $i\in \{1,\ldots, k\}$ such that $\|\chi_Ap_{\xi_i}\|\geq 7\pi\eps$. Using that $\|\chi_A\xi_i\|=\|\chi_Ap_{\xi_i}\|$ and  that $k$-randomly picked vectors in $\partial B_{V}$ are asymptotically orthonormal as $n\to \infty$, we have   
  \begin{align*}
    \mathbb P\left(\bar\xi\in (\partial B_{V})^k\mid \max_{\substack{  A\subseteq \{1,\ldots, n\}\\ |A|\leq n \delta}} \|\chi_Ap_{\bar\xi}\|\geq  7\pi\eps\sqrt{k}\right)\leq   L  e^{-cn},
    \end{align*}
for all sufficiently large $n$, where $L$ and $c$ are constants depending only on $k$, $\delta$, and $\gamma$.
\end{proof}

We can now obtain the strengthening of Theorem B in \cite{Ozawa2023uRaSmallerQL} needed for our goals.
\begin{theorem}\label{ThmExpContainsPropducMatrixPropDim}
    Let $X=\bigsqcup_{n\in\mathbb{N}}X_n$ be the coarse disjoint union of expander graphs and $0<\gamma\leq 1$. For each $n\in \N$, let $V_n\subseteq \ell_2(X_n)$ be a subspace of dimension at least $|X_n|\gamma$. Then, there are an increasing sequence $(n(k))_{k\in\mathbb{N}}\subseteq\mathbb{N}$ and a sequence $(W_k)_{k\in\mathbb{N}}$ such that
    \begin{enumerate}
        \item each $W_k$ is a subspace of $V_{n(k)}$,
        \item $\lim_{k\to\infty}\dim(W_k)=\infty$, and 
        \item $\prod_{k\in\mathbb{N}}\cB(\ell_2(W_k))\subseteq \cstql(X)$.
    \end{enumerate}
\end{theorem}

\begin{proof}
 Let $(\delta_k)_{k\in\mathbb{N}}$ be a decreasing sequence of positive reals converging to $0$, such that $\delta_0<\gamma$, and $\lim_{k\to\infty}\eps_k=0$, where 
  \[\eps_k=14\pi\sqrt{(\delta_k/\gamma)\log(1/\delta_k)k}\]
  for $k\in\N$. Since $\lim_{n\to\infty}|X_n|=\infty$, Lemma \ref{LemmaMedellin12Dec23} allows us to find a strictly increasing sequence $(n(k))_{k\in\mathbb{N}}$ in $\N$ such that for each $k\in\N$, there is a subspace  $W_k\subseteq V_{n(k)}$ with dimension $k$ such that, letting $p_k$ be the orthogonal projection of  $\C^{|X_{n(k)}|}$ onto $W_k$, we have
\begin{equation}\label{EqPierdeMuchaRelevancia}\max\left\{\|\chi_Ap_k\|\mid A\subseteq \{1,\ldots, |X_{n(k)}|\},\ |A|\leq \delta_m|X_{n(k)}|\right\}\leq\eps_m. \end{equation}
  for all $m\in \{1,\ldots, k\}$.

  Let us show that $\prod_{k\in\mathbb{N}}\cB(W_k)$ is contained in $\cstql(X)$. For that, let $a=\SOTh \sum_ka_k\in  \prod_{k\in\mathbb{N}}\cB(W_k) $, where each $a_k$ is an operator in $\cB(W_k)$, and let us show that $a$ is quasi-local. Without loss of generality, assume $\|a\|\leq 1$, so $\|a_k\|\leq 1$ for all $k\in\N$. Fix $\eps>0$ and pick   $m\in\N$ such that $\eps_m<\eps $.  As $X=\bigsqcup_{n\in\mathbb{N}} X_n$ is a coarse disjoint union of expander graphs,   there is  $\kappa>1$ such that for all $n\in\N$ and all $A,B\subseteq X_n$ we have 
  \begin{equation}
      \label{EqExpGapMed12dec23}
      \min\left\{\frac{|A|}{|X_n|}, \frac{|B|}{|X_n|}\right\}\leq \kappa^{-\frac{d(A,B)}{2}};
  \end{equation}
  see \cite[Section 3]{Ozawa2023uRaSmallerQL}, and \cite{Li:2021aa, KhukhroLiVigoloZhang2021AdvMath} for more details about this type of condition.
  Therefore, we can   choose $r>0$ large enough so that for all $n\in\N$ and all $A,B\subseteq X_n$ with $d(A,B)>r$,  
  \begin{equation}\label{EqUltimoDiaMedellin23}
    \min\left\{\frac{|A|}{|X_n|}, \frac{|B|}{|X_n|}\right\}\leq \delta_{m}. 
  \end{equation}
  Fix $A,B\subseteq X$ with $d(A,B)>r$, and for each $k\in\N$ let $A_k=A\cap X_{n(k)}$ and $B_k=B\cap X_{n(k)}$. Thus
  \[\|\chi_Aa\chi_B\|=\sup_{k\in\N}\|\chi_{A_k}a_k\chi_{B_k}\|.\]
  Note that $\|\chi_{A_k}a_k\chi_{B_k}\|\leq\eps$ for all $k\geq m$. Indeed, fix such $k$. Since $d(A,B)>r$, we also have  $d(A_k,B_k)>r$. Hence, by \eqref{EqUltimoDiaMedellin23}, either 
$|A_k|\leq \delta_{m}|X_{n(k)}|$ or $|B_k|\leq \delta_{m}|X_{n(k)}|$. Without loss of generality, assume the former happens. Therefore, as $k\geq m$, \eqref{EqPierdeMuchaRelevancia} implies that $\|\chi_{A_k}p_k\|\leq \eps_m$. Since $p_k$ is the identity of $\cB(W_k)$ and $\|a_k\|\leq 1$, we have 
\[
\|\chi_{A_k}a_k\chi_{B_k}\|\leq \|\chi_{A_k}p_k\|\leq \eps_m.
\]
As $\eps_m<\eps$ and $k\geq m$ was arbitrary, we have shown that
\[
\|\chi_{A_k}a_k\chi_{B_k}\|\leq \eps\ \text{ for all }\ k\geq m.
\]
Since $\sum_{k=1}^ma_k$ is compact, this finishes the proof. 
\end{proof}

The next ingredient for our proof of the main Theorem \ref{ThmExpNoCartanlInfty} is the following:
\begin{lemma}\label{lemma:ueqmasas}
Let $X=\bigsqcup_{n\in\mathbb{N}} X_n$ be the coarse disjoint union of finite metric spaces. For $F\subseteq\mathbb{N}$, denote $X_F=\bigsqcup_{n\in F}X_n$. Suppose that $A\subseteq \cstql(X)$ is Cartan subalgebra   isomorphic to $\ell_\infty$. Then there exists a unitary $u\in\B(\ell_2(X))$ such that $u-1$ is compact (and thus $u\in\cstql(X)$), and for every $F\subseteq\mathbb N$ we have $\chi_{X_F}\in uAu^*$.
\end{lemma}

For its proof, we need to recall some properties of Higson functions. Let $(X,d)$ be a u.l.f.\ metric space. A bounded function $f\colon X\to \mathbb C$ is called a \emph{Higson function} if for all $\varepsilon>0$ and $r>0$ there is a finite $Z\subseteq X$ such that for all $x,y\notin Z$ with $d(x,y)<r$ we have that $|f(x)-f(y)|<\varepsilon$.  Higson functions form a subalgebra of $\ell_\infty(X)$, which we denote by $C_h(X)$. Letting 
\[
\pi\colon \cB(\ell_2(X))\to \cB(\ell_2(X))/\cK(\ell_2(X))
\]
be the canonical quotient map of $\cB(\ell_2(X))$ onto its Calkin algebra, we have that 
\begin{equation}\label{Eq.HigsonFunctions.Center.Corona}
 \pi[C_h(X)]=\mathcal Z(\cstql(X)/\mathcal K(\ell_2(X)))=\mathcal Z(\cstu(X)/\mathcal K(\ell_2(X)))
\end{equation}
by \cite[Proposition 3.6]{BaudierBragaFarahVignatiWillett2024vNa}, where $\mathcal Z(C)$ denotes the center of a C*-algebra $C$. 

Let now $X=\bigsqcup_{n\in\mathbb{N}} X_n$ be the coarse disjoint union of  metric spaces. For any $F\subseteq\mathbb{N}$,  the function $\chi_{X_F}$ is a Higson function, since $d(X_n,X_m)\to \infty$ as $n,m\to \infty$ with $n\neq m$. In particular, we have that 
\begin{equation}
    \pi(\chi_{X_F})\in \mathcal Z(\cstql(X)/\mathcal K(\ell_2(X)))
\end{equation}
for all $F\subseteq \N$.

Before proving Lemma \ref{lemma:ueqmasas}, we recall two well-known lemmas.

\begin{lemma}\label{LemmaBlack:Operator}
If $p$ and $q$ are projections in a unital \cstar-algebra $A$ with $\norm{p-q}<1$, then there is a unitary $u\in A$ with $upu^*=q$ and $\norm{1-u}<\sqrt{2}\norm{p-q}$.
\end{lemma}

\begin{remark}
    A conclusion with a less transparent estimate than $\sqrt{2}\norm{p-q}$ appears in almost every C*-algebra textbook, for example \cite[Proposition II.3.3.4(ii)]{Black:Operator}. The above version follows from the discussion before Theorem V.1.41 in \cite{Tak:TheoryI}; for further details see a MathOverflow answer \cite{mathoverflow:close-projns-are-unitarily-equiv}. In fact, one can use $u=(p+q-1)|p+q-1|^{-1}(2p-1)$.
\end{remark}

The next lemma follows from straightforward calculations, so we omit its proof.

\begin{lemma}\label{L.sumoforthogonals}
Let $m\in\NN$, $\varepsilon>0$, and let $(p_i)_{i=1}^m$ and $(q_i)_{i=1}^m$ be sequences of mutually orthogonal projections in $\mathcal B(H)$. For each $i\leq m$, suppose   $v_i$ is a partial isometry with domain $p_i[H]$ and codomain $q_i[H]$ such that $\norm{v_i\xi-\xi}<\varepsilon\norm{\xi}$ for every $\xi\in p_i[H]$. Then for every $\xi\in (\sum_{i=1}^m  p_i)[H]$ we have that $\norm{\sum_{i=1}^m v_i\xi-\xi}<\varepsilon\norm{\xi}$.
\end{lemma}

\begin{proof}[Proof of Lemma \ref{lemma:ueqmasas}]
As $\cstql(X)$ contains the compacts, it follows from \cite[Lemma 2.5 and Proposition 2.7]{WhiteWillett2017} that there is a  sequence  $(q_n)_{n\in\mathbb{N}}$ in $A$ of orthogonal projections with rank 1, such that  \[A=W^*(\{q_n\mid n\in\N\}),\] i.e., $A$ is the von Neumann algebra generated by $(q_n)_{n\in\mathbb{N}}$. As $(q_n)_{n\in\mathbb{N}}$ is maximal with this property, we have $\SOTh\sum_{n\in\mathbb{N}}q_n=\Id_{\ell_2(X)}$, and $A$ is a masa (maximal abelian C*-subalgebra) in $\B(\ell_2(X))$.

For each $G\subseteq \N$, let    
\[
q_G=\sum_{n\in G}q_n.
\]

\begin{claim}\label{claim1}
For every $F\subseteq\mathbb N$ we have that $\pi(\chi_{X_F})\in \pi[A]$. Hence for every infinite $F\subseteq\mathbb N$ there is $G\subseteq\mathbb{N}$ such that $\chi_{X_F}-q_G$ is compact.
\end{claim}

\begin{proof}
First, the image of a masa in $\cB(\ell_2(X))$ under $\pi$ is a masa in the Calkin algebra $\cB(\ell_2(X))/\cK(\ell_2(X))$. Indeed,  this is a consequence of a theorem by Johnson and Parrott (see  \cite[Theorem 2.1]{JohnsonParrott1972JFA}); alternatively, see \cite[Theorem 3.5]{BaudierBragaFarahVignatiWillett2024vNa} for a precise derivation of this using  Johnson and Parrott's theorem. Therefore   $\pi[A]$ is a masa in the Calkin algebra, and thus it is a masa in $\cstql(X)/\mathcal K(\ell_2(X))$. Since $\pi(\chi_{X_F})\in\mathcal Z(\cstql(X)/\mathcal K(\ell_2(X)))$ by \eqref{Eq.HigsonFunctions.Center.Corona}, each masa of $\cstql(X)/\mathcal K(\ell_2(X))$ must contain $\pi(\chi_{X_F})$ for all $F\subseteq \N$. In particular, $\pi[A]$ does.

The second assertion follows from the fact that all projections in $\pi[A]$ have the form $\pi(q_G)$ for some $G\subseteq\mathbb N$. Indeed, as $A=W^*(\{q_n\in n\in\N\})$, $A$ is a von Neumann algebra and, in particular, it has real rank zero. This implies that projections in $\pi[A]$ must be the image of projections in $A$ (see, for instance,  \cite[Lemma 3.1.13]{FarahBook2019} for a proof).
\end{proof}

\begin{claim}\label{claim2}
For every $\varepsilon>0$ there is $n\in\N$ such that for every finite $F\subseteq\mathbb N$ with $\min F>n$ there exists a (necessarily finite) $G\subseteq\mathbb N$ such that $\norm{\chi_{X_F}-q_G}<\varepsilon$.
\end{claim}

\begin{proof}
Suppose this is not the case for a given $\eps>0$. Then there is a disjoint sequence $(F_k)_{k\in\mathbb{N}}$ of  finite subsets of $\N$  such that  $\|\chi_{X_{F_k}}-q_G\|\geq \varepsilon$ for all $k\in\N$ and all $G\subseteq \N$. Applying \cite[Lemma 4.3]{BaudierBragaFarahVignatiWillett2023BijExp} to the sequences $(p_k=\chi_{F_k})_{k\in\mathbb{N}}$ and $(q_k)_{k\in\mathbb{N}}$, we get that there is a (necessarily infinite) $F\subseteq\mathbb N$ such that $\chi_{X_F}-q_G$ is not compact for any $G\subseteq\mathbb N$. This contradicts Claim~\ref{claim1}.
\end{proof}

\begin{claim}\label{claim3}
There is an increasing sequence $(n_k)_{k\in\mathbb{N}}$ of naturals   such that for every $k\in\N$ and every    $F\subseteq\mathbb N$ with $\min F>n_k$ there is   $G=G(F)\subseteq \N$ such that
\[
\norm{\chi_{X_F}-q_G}\leq 2^{-k+1}.
\]
If $k>2$, $G$ is unique. \end{claim}

\begin{proof}
We let $n_{0}=0$ and construct $(n_k)_{k\in\mathbb{N}}$ by induction. If $n_k$ has been defined, let $n_{k+1}>n_k$ be given by Claim~\ref{claim2} for $\varepsilon=2^{-k}$. Suppose now that $F$ is such that $\min F>n_k$. If $F$ is finite, the existence of $G$ is guaranteed by Claim~\ref{claim2} directly. If $F$ is infinite, let 
\[
G=\bigcup_{F'\subseteq F,\ |F'|<\infty}G(F').
\]
So, \[q_G=\SOTh \lim_{F'\subseteq F,\ |F'|<\infty} q_{G(F')}. \]
Since 
\[
\chi_{X_F}=\SOTh\lim_{F'\subseteq F,\ |F'|<\infty} \chi_{X_{F'}}  ,
\]
we have that 
\[
\norm{\chi_{X_F}-q_G}\leq\limsup_{F'\subseteq F, F' \text{ finite}}\norm{\chi_{X_{F'}}-q_{G(F')}}\leq 2^{-k+1}.
\]
The uniqueness of $G(F)$ for $\min F>n_3$ comes from the fact that   $\norm{q_G-q_{G'}}=1$ for any distinct $G$ and $G'$. 
\end{proof}
We now continue with the proof of Lemma \ref{lemma:ueqmasas}. Fix $\tilde F=\mathbb N\setminus \{0,\ldots,n_3\}$, and let $H=G(\tilde F)$. Since 
\[
\norm{(1-\chi_{X_{\tilde F}}) - (1-q_H)}= \norm{\chi_{X_{\tilde F}}-q_H}\leq \frac{1}{2},
\]
the projections $1-\chi_{X_{\tilde F}}$ and $1-q_H=q_{\mathbb N\setminus H}$ have the same finite rank.
Let $v_0$ be a partial isometry such that $v_0v_0^*=\chi_{X_{\{0,\ldots,n_3\}}}$ and $v_0^*v_0=q_{\mathbb N\setminus H}$ with the property that for every $n\leq n_3$ we have that $v_0^*\chi_{X_n}v_0$ is of the form $q_F$ for some $F\subseteq \mathbb N\setminus H$. This implies that $v_0q_{\mathbb N\setminus H}v_0^*=\chi_{X_{\{0,\ldots,n_3\}}}$.

Suppose $i\geq 3$ and let $n\in (n_i,n_{i+1}]$. Since $\|\chi_{X_n}-q_{G(n)}\|<2^{-i+1}$ we can find a unitary $u_n\in\mathcal B(\ell_2(X))$ such that $\|u_n-1\|<\sqrt{2}2^{-i+1}$ such that $u_nq_{G(n)}u_n^*=\chi_{X_n}$ by Lemma \ref{LemmaBlack:Operator}. We let $v_n=\chi_{X_n}u_nq_{G(n)}$ so that $v_n$ is a partial isometry, $v_nv_n^*=\chi_{X_n}$ and $v_n^*v_n=q_{G(n)}$ for all $n$. 

By considering (for example) ranks of the projections in the orthogonal family $(\chi_{X_n})_{n>n_3}$, we observe that the projections $(q_{G(n)})_{n>n_{3}}$ are also pairwise orthogonal. Furthermore,  since $\SOTh\sum_{n>n_3}\chi_{X_n} = \chi_{\tilde F}$, we have 
 \[\SOTh\sum_{n>n_3}q_{G(n)} = q_H.\]

Consequently, the formula \[u=v_0+\sum_{n>n_3} v_n\] defines a unitary.
Furthermore, by construction, $\chi_{X_F}=uq_{G(F)}u^*\in uAu^*$ for all $F\subseteq \mathbb{N}$.

We are left to show that $u-1$ is compact. Pick $\varepsilon>0$, and let $k$ be such that $2^{-k+1}<\varepsilon$. Let $G=\bigcup_{n\leq n_k}G(n)$. Let $\xi$ be a unit vector in the orthogonal of $q_{G}$. We can assume that $\xi$ has finite support, so that there is a finite $k'$ such that $\xi\in q_S[H]$ where $S=\bigcup_{n\in (n_k,n_{k'}]}G(n)$. In particular, $u\xi=\sum_{n\in (n_k,n_{k'}]}v_n\xi$. Since for all such $n$ we have that $\norm{v_n\xi-\xi}<\sqrt{2}\varepsilon$ and $\xi$ is a unit vector, Lemma~\ref{L.sumoforthogonals} gives that $\norm{u\xi-\xi}<\sqrt{2}\varepsilon$. As $\varepsilon$ is arbitrary, this shows that $u-1$ is compact.\end{proof}

\begin{proof}[Proof of Theorem \ref{ThmExpNoCartanlInfty}] For a contradiction, suppose that there is a Cartan subalgebra $A\subseteq \cstql(X)$ isomorphic to $\ell_\infty$. Applying Lemma~\ref{lemma:ueqmasas}, we can assume that each $\chi_{X_F}$ belongs to $A$ for all $F\subseteq\mathbb{N}$. By \cite[Proposition 4.15]{WhiteWillett2017}, there is a u.l.f.\ coarse space\footnote{$Y$ is not necessarily a metric space, only a coarse space (see e.g.~\cite[Section 4]{WhiteWillett2017} or \cite{RoeBook} for details), as we do not assume co-separability of $A$. See also Remark \ref{Remark.metrizability} below.} $(Y,\cE)$ and an isomorphism $\Phi\colon \cstql(X)\to \cstu(Y)$ such that $\Phi[A]=\ell_\infty(Y)$. By \cite[Corollary 3.3]{BaudierBragaFarahKhukhroVignatiWillett2021uRaRig} applied to the collection $\left(\Phi^{-1}(\chi_y)\right)_{y\in Y}$, there is $\theta>0$ and $f\colon X\to Y$ such that 
\[\|\Phi(\chi_x)\delta_{f(x)}\|\geq \theta\]
for all $x\in X$.\footnote{Note that while \cite[Corollary 3.3]{BaudierBragaFarahKhukhroVignatiWillett2021uRaRig} is stated for the uniform Roe algebra, it also holds for the quasi-local algebra. This is explicitly noted in \cite[Theorem 3.5]{BaudierBragaFarahKhukhroVignatiWillett2021uRaRig}. } Moreover, the map $f$ is a coarse embedding (see \cite[Theorem 1.12]{BaudierBragaFarahKhukhroVignatiWillett2021uRaRig}). In particular, $Z=f[X]$ is metrizable.

For each $n\in\N$, let $Z_n=f[X_n]$ and  $p_n=\Phi^{-1}(\chi_{Z_n})$.    Notice that, since $f$ is uniformly finite-to-one, there is $\gamma>0$ such that 
\begin{equation}\label{EqZnproportionalXn}
|Z_n|\geq |X_n|\gamma \ \text{  for all }\ n\in\N.
\end{equation} 
If $z\in Z_n$, let $x$ be such that $z=f(x)$. We thus have that \[\norm{\Phi^{-1}(\chi_z)\chi_{X_n}}\geq \norm{\Phi^{-1}(\chi_z)\chi_x}\geq\theta>0.\] Since $\Phi^{-1}(\chi_z)$ and $\chi_{X_n}$ commute (as they both belong to $A$) and $\Phi^{-1}(\chi_z)$ has rank one, then $\Phi^{-1}(\chi_z)\leq\chi_{X_n}$. As $z$ is arbitrary, then $p_n\leq\chi_{X_n}$.

 Letting $V_n=\mathrm{range}(p_n)$ for all $n\in\N$, we have that each $V_n$ is a subspace of $\ell_2(X_n)$ with dimension at least $|X_n|\gamma$. By Theorem \ref{ThmExpContainsPropducMatrixPropDim}, there  are an increasing sequence $(n_k)_{k\in\mathbb{N}}$ of naturals and a sequence $(W_k)_{k\in\mathbb{N}}$ such that
    \begin{enumerate}
        \item each $W_k$ is a subspace of $V_{n_k}$,
        \item $\lim_{k\to\infty}\dim(W_k)=\infty$, and 
        \item $\prod_{n\in\mathbb{N}}\cB(\ell_2(W_k))\subseteq \cstql(X)$.
    \end{enumerate}
Since $W_k$ is a subspace of $V_{n_k}$, 
\[
\Phi\restriction \mathcal B(\ell_2(W_k))\subseteq\chi_{Z_{n_k}}\cstu(Y)\chi_{Z_{n_k}}\subseteq\chi_Z\cstu(Y)\chi_Z=\cstu(Z).
\] Since $\dim W_k\to\infty$, this provides a copy of $\prod_{k\in\mathbb{N}}\mathrm M_{m_k}(\mathbb C)$, for some increasing sequence $m_k$, inside $\cstu(Z)$, the uniform Roe algebra of the metrizable space $Z$. This contradicts \cite[Theorem A]{Ozawa2023uRaSmallerQL} and finishes the proof.
\end{proof}

\begin{remark}\label{Remark.Asymp.Exp}
We point out that the only property of expander graphs used in the proof of Theorem \ref{ThmExpContainsPropducMatrixPropDim} above is \eqref{EqExpGapMed12dec23}. This is in fact a weaker condition than being an expander and it is referred to in the literature as an \emph{asymptotic expander}, see \cite{KhukhroLiVigoloZhang2021AdvMath}. Therefore, Theorem \ref{ThmExpNoCartanlInfty} and, in particular, Theorem \ref{ThmExpNoCartanlInfty} remain valid for asymptotic expanders as well. This is also the case in \cite{Ozawa2023uRaSmallerQL}.
\end{remark}

\begin{remark}
    \label{Remark.metrizability}
It remains open whether metrizability of coarse spaces is a property stable under isomorphisms. More precisely, suppose that $X$ is a u.l.f.\ metric space and $Y$ is a u.l.f.\ coarse space (see \cite{RoeBook} for more details about coarse spaces). If either $\cstu(X)\cong \cstu(Y)$ or $\cstql(X)\cong \cstql(Y)$, does it follow that $Y$ is metrizable? If this question had a positive answer, then Theorem \ref{ThmExpNoCartanlInfty} would become much more straightforward. 
\end{remark}

\section{Extension of isomorphisms between uniform Roe algebras}\label{SectionExtIso}

As mentioned in the introduction, the question whether $\cstu(X)\cong \cstu(Y)$ if and only if $\cstql(X)\cong \cstql(Y)$ remains open. In this subsection, we prove that the forward implication is always valid.

\begin{proof}[Proof of Theorem \ref{ThmIsoOfURAExtendsToQL}]
    Let $\Phi\colon \cstu(X)\to \cstu(Y)$ be an isomorphism.   By \cite[Lemma 3.1]{SpakulaWillett2013AdvMath} $\Phi$ is spatially implemented, meaning there is a unitary $u\colon \ell_2(X)\to \ell_2(Y)$ such that $\Phi=\mathrm{Ad}(u)$. So, $\Phi$ extends to an isomorphism between $\cB(\ell_2(X))$ and $\cB(\ell_2(Y))$, which we still denote by $\Phi$. We show that it restricts to an isomorphism from $\cstql(X)$ to $\cstql(Y)$.

    Since $\Phi$ restricts to an isomorphism between $\cK(\ell_2(X))$ and $\cK(\ell_2(Y))$, it induces an isomorphism 
    $\Psi$ between the Calkin algebras of $\ell_2(X)$ and $\ell_2(Y)$ making the following diagram commute.
\[
\xymatrix{ \cB(\ell_2(X))\ar[r]^\Phi\ar[d]_{\pi_X}& \cB(\ell_2(X))\ar[d]^{\pi_Y}\\
\cB(\ell_2(X))/\cK(\ell_2(X))\ar[r]_\Psi& \cB(\ell_2(X))/\cK(\ell_2(Y)).}
\] 
The isomorphism $\Psi$ restricts to an isomorphism between $\cstu(X)/\cK(\ell_2(X))$  and  $\cstu(Y)/\cK(\ell_2(Y))$, which then restricts to an isomorphism between their centers. As recalled above in \eqref{Eq.HigsonFunctions.Center.Corona}, the centers of these algebras are equal to the corresponding Higson coronas. Hence, modulo compact operators, isomorphisms between uniform Roe algebras send Higson functions to Higson functions.
    
    Therefore, the fact that $\Phi$ restricts to an isomorphism between $\cstql(X)$ and $\cstql(Y)$ follows immediately from \cite[Theorem 3.3]{SpakulaZhang2020JFA}: given a u.l.f.\ metric space $Z$, an operator $a\in \cB(\ell_2(Z))$ is in $\cstql(Z)$ if and only if $[a,h]$ is compact for all Higson functions $h\colon Z\to \C$.
\end{proof}

\begin{remark}\label{remark1}
    It remains open whether an isomorphism between $\cstql(X)$ and $\cstql(Y)$ implies that $\cstu(X)$ and $\cstu(Y)$ are isomorphic. On the other hand, this is known in case of coarse disjoint unions of expander graphs. In fact, in this case, isomorphism of quasi-local algebras implies bijective coarse equivalence of the underlying spaces (this follows from the main result of \cite{BaudierBragaFarahVignatiWillett2023BijExp} and considerations made in \cite{martnez2024rigidity}), and thus isomorphism of their uniform Roe algebras. We mention that it is known that an isomorphism  $\cstql(X)\to \cstql(Y)$ does not need to restrict to an isomorphism  $\cstu(X)\to \cstu(Y)$. Indeed, this is done in \cite[Remark 1.3.2.(iv)]{martnez2024rigidity} for the non-uniform versions of the Roe algebra and the quasi-local algebra, and it is straightforward to adapt the argument for the uniform versions.
\end{remark}

\begin{acknowledgments}
    This paper was written under the auspices of the American Institute of Mathematics (AIM) SQuaREs program and as part of the \emph{Expanders, ghosts, and Roe algebras} SQuaRE project. B. M. Braga  was partially supported by FAPERJ, grant E-26/200.167/2023,  by CNPq, grant 303571/2022-5, and by Serrapilheira, grant R-2501-51476. J.~\v{S}pakula was partially supported by UK's Engineering and Physical Sciences Research Council (EPSRC) Standard Grant EP/V002899/1. A. Vignati was partially supported by the Institut Universitaire de France. The authors would like to thank Ilijas Farah and Rufus Willett for conversations and support in the making of this paper, and Kostyantyn Krutoy for useful comments on Remark \ref{remark1}.
\end{acknowledgments}
 
\bibliographystyle{amsalpha}
 \bibliography{bibliography}

\end{document}